\documentclass[11pt,reqno]{amsart}
\usepackage{amsmath,amsxtra,latexsym,amsthm,amssymb,amscd,pb-diagram}
\usepackage{mathabx,mathrsfs,amsfonts}
\usepackage[colorlinks=true,linkcolor=magenta,citecolor=blue]{hyperref}
\usepackage[margin=1in]{geometry}

\usepackage{graphicx}
\usepackage{epsfig}
\usepackage{color}
\usepackage{enumitem}

\newtheorem{dn}{Definition}[section]
\newtheorem{dl}{Theorem}[section]
\newtheorem{md}{Proposition}[section]
\newtheorem{proposition}{Proposition}[section]
\newtheorem{bd}{Lemma}[section]
\newtheorem{lemma}{Lemma}[section]
\newtheorem{hq}{Corollary}[section]
\newtheorem{nx}{Remark}[section]
\newtheorem{remark}{Remark}[section]
\newtheorem{vd}{Example}[section]
\newcommand{\R}{\mathbb{R}}
\newcommand{\Z}{\mathbb{Z}}

\newcommand{\ity}{\infty}

\newcommand{\bbd}{\begin{bd}}
\newcommand{\ebd}{\end{bd}}
\newcommand{\bdn}{\begin{dn}}
\newcommand{\edn}{\end{dn}}
\newcommand{\bhq}{\begin{hq}}
\newcommand{\ehq}{\end{hq}}
\newcommand{\bdl}{\begin{dl}}
\newcommand{\edl}{\end{dl}}
\newcommand{\bnx}{\begin{nx}}
\newcommand{\enx}{\end{nx}}
\newcommand{\bmd}{\begin{md}}
\newcommand{\emd}{\end{md}}
\newcommand{\bvd}{\begin{vd}}
\newcommand{\evd}{\end{vd}}

\title[Blow up results for structural damped wave model with nonlinear memory]{Blow up results for semi-linear structural damped wave model with nonlinear memory}
\author{Tuan Anh Dao}
\address{Tuan Anh Dao \hfill\break
$\quad$ School of Applied Mathematics and Informatics, Hanoi University of Science and Technology, No.1 Dai Co Viet road, Hanoi, Vietnam \hfill\break
Faculty for Mathematics and Computer Science, TU Bergakademie Freiberg, Pr\"{u}ferstr. 9, 09596, Freiberg, Germany}
\email{anh.daotuan@hust.edu.vn}

\author{Ahmad Z. Fino}
\address{Ahmad Z. Fino \hfill\break
Department of Mathematics, Faculty of Sciences, Lebanese University, P.O. Box 826, Tripoli, Lebanon}
\email{ahmad.fino01@gmail.com; afino@ul.edu.lb}

\begin{document}

\begin{abstract}
This article is to study the nonexistence of global solutions to semi-linear structurally damped wave equation with nonlinear memory in $\R^n$ for any space dimensions $n\ge 1$ and for the initial arbitrarily small data being subject to the positivity assumption. We intend to apply the method of a modified test function to establish blow-up results and to overcome some difficulties as well caused by the well-known fractional Laplacian $(-\Delta)^{\sigma/2}$ in structural damping terms.
\end{abstract}

\maketitle

\section{Introduction} \label{Sec.Intro}
Main goal of this paper is concerned with the following Cauchy problem for semi-linear structurally damped wave equation with nonlinear memory
\begin{equation} \label{1}
\begin{cases}
u_{tt}- \Delta  u+ \mu(-\Delta)^{\sigma/2} u_t= \displaystyle \int_0^t (t-s)^{-\gamma} |u(s)|^p\, ds, &\quad x\in \R^n,\, t> 0, \\
u(0,x)= u_0(x),\quad u_t(0,x)=u_1(x), &\quad x\in \R^n,
\end{cases}
\end{equation}
where $\mu>0$, $\sigma \in (0,2)$, for some $\gamma \in (0,1)$ and $p>1$.
\par Taking into considerations the first limit case of the parameter $\sigma$, we want to refer the reader to some previous results on a typical important problem of (\ref{1}) with $\sigma=0$, the so-called semi-linear classical damped wave equations. In particular, the Cauchy problem for the damped wave equation with nonlinear memory was considered in the pioneering paper of Fino \cite{Fino2011}. With the suitable choice of $\gamma \in (0,1)$, the author succeeded to prove the global existence of small data solutions in the low space dimensions $1\le n\le 3$ by using the weighted energy method and the blow-up result for any dimensional space by the test function method as well. After Fino \cite{Fino2011}, the authors in \cite{Dabbicco1,YangaShiaZhu} developed his results in several different approaches. More in detail, in \cite{YangaShiaZhu} Yanga-Shia-Zhu improved the global result by removing the compactness of the support on the initial data. D'Abbicco in \cite{Dabbicco1} has made suitably different choice of initial data spaces and solution spaces depending on each distinguished range of $\gamma \in (0,1)$ to extend the global existence results of Fino in space dimensions $n\le 5$. Thereafter, involving the scenario of (\ref{1}) with $\sigma=1$ in \cite{Dabbicco2} D'Abbicco has studied the following semi-linear wave equation with structural damping and nonlinear memory:
\begin{equation} \label{2}
\begin{cases}
u_{tt}- \Delta  u+ \mu(-\Delta)^{1/2} u_t= \displaystyle \int_0^t (t-s)^{-\gamma} |u(s,\cdot)|^p\, ds, &\quad x\in \R^n,\, t> 0, \\
u(0,x)= u_0(x),\quad u_t(0,x)=u_1(x), &\quad x\in \R^n.
\end{cases}
\end{equation}
Thanks to the special structure of the corresponding linear problem, the global existence results for (\ref{2}) were obtained in any dimensional cases $n\ge 2$. Moreover, a counterpart result for nonexistence of global solutions was also indicated to find the following critical exponent:
$$ p_{crit}(n,\gamma):= \max\big\{p_\gamma(n),\gamma^{-1}\big\}, \quad \text{ where }p_\gamma(n):=1+\frac{3-\gamma}{n-2+\gamma} $$
by the application of test function method and a maximum principle (see more \cite{DabbiccoReissig}). For the purpose of further considerations, the author in the cited paper discussed several results for the global existence of small data solutions to (\ref{1}) regarding not only the general cases $\sigma \in (0,2)$, but also the other limit case $\sigma= 2$. However, it still remains an open problem as far as to show nonexistence results for (\ref{1}) in the general cases. From this observation, the main novelty of this paper is to look for these results in any space dimensions.
\par The essential difficulty to investigate the nonexistence of global solutions to (\ref{1}), where $\sigma$ is assumed to be a fractional number $\in (0,2)$, is to deal with the fractional Laplacian $(-\Delta)^{\sigma/2}$, the so-called nonlocal operators, in general. This difficulty does not happen in the special case $\sigma=1$ appearing in \cite{Dabbicco2,DabbiccoReissig} by the aid of the nonnegativity of fundamental solutions, which cannot be expected for any $\sigma \in (0,2)$. As we can see, the authors in the cited papers restrict themselves in the case where the first data $u_0=0$ and the second data $u_1$ is non-negative. This restriction leads them to the nonnegativity of solutions which comes into play, via Ju's inequality \cite{Ju} or \cite[Appendix]{Finokirane}, in the proofs of nonexistence of global solutions. Unfortunately, the method used in \cite{Dabbicco2,DabbiccoReissig} is not so well-working when we want to discuss the case of possibly sign-changing data $u_1$ or even if $u_0$ is not identically zero. Here we want to point out that our main result in this paper is not only to cover D'Abbicco's result in \cite{Dabbicco2}, but also to overcome the above mentioned difficulties. Quite recently, Dao-Reissig in \cite{DaoReissig} have succeeded to prove blow-up results to determine the critical exponents for the following Cauchy problem for semi-linear structurally damped $\sigma$-evolution models:
$$ \begin{cases}
u_{tt}+ (-\Delta)^\sigma u+ (-\Delta)^{\delta} u_t=|u|^p, &\quad x\in \R^n,\, t> 0, \\
u(0,x)= u_0(x),\quad u_t(0,x)=u_1(x), &\quad x\in \R^n,
\end{cases} $$
where $\sigma \ge 1$ and $\delta\in [0,\sigma)$ are any fractional numbers, by using a modified test function method. In this connection, our key tool used in this paper is strongly motivated by the paper \cite{DaoReissig}. More precisely, another modified test function method (see \cite{BonforteVazquez}) can be applied to catch the desired nonexistence results for (\ref{1}) in any space dimensions. \medskip

{\bf Notations} \\
Throughout the present paper, we use the following notations.
\begin{itemize}[leftmargin=*]
\item For any $r \in \R$, we denote $[r]_+:= \max\{r,0\}$ as its positive part, and $[r]:= \max \big\{k \in \Z \,\, : \,\, k\le r \big\}$ as its integer part.
\item For later convenience, we put
$$p_c:=1+\frac{2+(1-\gamma)(2-\tilde{\sigma})}{\big[n-2+\gamma(2-\tilde{\sigma})\big]_+},$$
where $\tilde{\sigma}:=\min\{\sigma,1\}$.
\item Hereafter $C$ denotes a suitable positive constant and may have different value from line to line.
\item We write $f\lesssim g$ when $f\le Cg$, and $f \approx g$ when $g\lesssim f\lesssim g$.
\end{itemize}

Our main result reads as follows.
\begin{dl}[\textbf{Main result}] \label{blow-up}
Let $\sigma \in (0,2)$ and $n \geq1$. We assume that the initial data $(u_0,u_1)\in H^\sigma(\mathbb{R}^n)\times L^2(\mathbb{R}^n)$ such that
$$u_0 \in L^1 \quad \text{ and }\quad (-\Delta)^{\sigma/2} u_0 + u_1 \in L^1,$$
satisfy the relation
\begin{equation} \label{optimal9.1}
\int_{\R^n} \big(u_1(x)+ \mu (-\Delta)^{\sigma/2} u_0(x)\big)dx >0.
\end{equation}
Moreover, we suppose one of the following conditions:
\begin{itemize}[leftmargin=*]
\item $p>1$ if $n\le [2-\gamma(2-\tilde{\sigma})]$ and for any $\gamma \in (0,1)$,
\item $p \in (1,p_c]$ if $[2-\gamma(2-\tilde{\sigma})]< n\le 2$ and for any $\gamma \in (0,1)$, or $n> 2$ and for any $\gamma \in \big[\frac{n-2}{n},1\big)$,
\item $p \in (1,\gamma^{-1})$ if $n\ge 3$ and for any $\gamma \in \big(0,\frac{n-2}{n}\big)$.
\end{itemize}
Then, there is no global (in time) weak solution to (\ref{1}).
\end{dl}
The remainder of this paper is organized as follows: Section \ref{Sec.Pre} is to collect some preliminaries. We devote to the proof of the main result in Section \ref{Sec.Proof}.

\section{Preliminaries} \label{Sec.Pre}
In this section, we present some definitions and results concerning the fractional integrals and
fractional derivatives that will be used hereafter.

\bdn (Absolutely continuous functions)\cite[Chapter~1]{SKM}\\
A function $f:[a,b]\rightarrow\mathbb{R}$ with $a,b\in\mathbb{R}$, is absolutely continuous if and only if there exists a Lebesgue summable function $\varphi\in L^1(a,b)$ such that 
$$f(t)=f(a)+\int_{a}^t\varphi(s)\,ds.$$ 
The space of these functions is denoted by $AC[a,b]$. Moreover, for all $m\geq0$, we define
$$AC^{m+1}[a,b]:= \big\{f:[a,b]\rightarrow\mathbb{R}\;\text{such that}\;D^m f\in
AC[a,b]\big\},$$
where $D^m=\frac{d^m}{dt^m}$ is the usual $m$ times derivative.
\edn

\bdn(Riemann-Liouville fractional integrals)\cite[Chapter~1]{SKM}\\
Let $f\in L^1(0,T)$ with $T>0$. The Riemann-Liouville left- and right-sided fractional integrals of order $\alpha\in(0,1)$ are, respectively, defined by
\begin{equation}\label{}
I^\alpha_{0|t}f(t):=\frac{1}{\Gamma(\alpha)}\int_{0}^t(t-s)^{-(1-\alpha)}f(s)\,ds, \quad t>0,
\end{equation}
and
\begin{equation}\label{}
I^\alpha_{t|T}f(t):=\frac{1}{\Gamma(\alpha)}\int_t^{T}(s-t)^{-(1-\alpha)}f(s)\,ds, \quad t<T,
\end{equation}
where $\Gamma$ is the Euler gamma function.
\edn

\bdn(Riemann-Liouville fractional derivatives)\cite[Chapter~1]{SKM}\\
Let $f\in AC[0,T]$ with $T>0$. The Riemann-Liouville left- and right-sided fractional derivatives of order $\alpha\in(0,1)$ are, respectively, defined by
\begin{equation}\label{}
D^\alpha_{0|t}f(t):=\frac{d}{dt}I^{1-\alpha}_{0|t}f(t)=\frac{1}{\Gamma(1-\alpha)}\frac{d}{dt}\int_{0}^t(t-s)^{-\alpha}f(s)\,ds, \quad t>0,
\end{equation}
and
\begin{equation}\label{}
D^\alpha_{t|T}f(t):=-\frac{d}{dt}I^{1-\alpha}_{t|T}f(t)=-\frac{1}{\Gamma(1-\alpha)}\frac{d}{dt}\int_t^{T}(s-t)^{-\alpha}f(s)\,ds, \quad t<T.
\end{equation}
\edn

\begin{proposition}(Integration by parts formula)\cite[(2.64)~p.46]{SKM}\\
Let $T>0$ and $\alpha\in(0,1)$. The fractional integration by parts formula
\begin{equation}\label{IP}
\int_{0}^{T}f(t)D^\alpha_{0|t}g(t)\,dt \;=\; \int_{0}^{T} 
g(t)D^\alpha_{t|T}f(t)\,dt
\end{equation}
is valid for every $f\in I^\alpha_{t|T}(L^p(0,T))$ and $g\in I^\alpha_{0|t}(L^q(0,T))$ such that $\frac{1}{p}+\frac{1}{q}\leq 1+\alpha$ with $p,q>1$, where
$$I^\alpha_{0|t}(L^q(0,T)):=\left\{f= I^\alpha_{0|t}h,\,\, h\in L^q(0,T)\right\}$$
and
$$I^\alpha_{t|T}(L^p(0,T)):=\left\{f= I^\alpha_{t|T}h,\,\, h\in L^p(0,T)\right\}.$$
\end{proposition}
\begin{remark}
A simple sufficiency condition for functions $f$ and $g$ to satisfy (\ref{IP}) is that $f,g\in C[0,T]$ such that
$D^\alpha_{t|T}f(t),D^\alpha_{0|t}g(t)$ exist at every point $t\in[0,T]$ and are continuous.
\end{remark}

\begin{proposition}\cite[Section~2.1]{KSTr}\\
Let $0<\alpha<1$ and $T>0$. Then, we have the following identities:
\begin{equation}\label{7}
    D^\alpha_{0|t}I^\alpha_{0|t}f(t)=f(t),\,\hbox{a.e. $t\in(0,T)$}, \quad\hbox{for all}\,f\in L^r(0,T) \text{ with } 1\leq r\leq\infty,
\end{equation}
and
\begin{equation}\label{5}
   (-1)^m D^m.D^\alpha_{t|T}f=D^{m+\alpha}_{t|T}f \quad\hbox{for all}\,f\in AC^{m+1}[0,T].
\end{equation}
\end{proposition}

Given $T>0$, let us define the function $w:\quad [0,T] \to \R$ by the following formula:
\begin{equation}\label{w}
\displaystyle w(t)=\left(1-t/T\right)^{\beta} \quad\hbox{for all}\,\,\,0\leq t\leq T,
\end{equation}
where $\beta \gg1$ is big enough. Later on, we need the following properties concerning the function $w$.
\begin{lemma}\cite[Property~2.1, p.71]{KSTr}\label{lemma1}\\
Let $T>0$, $0<\alpha<1$ and $m\geq0$. For all $t\in[0,T]$, we have
\begin{equation}\label{6}
D_{t|T}^{m+\alpha}
w(t)=\frac{\Gamma(\beta+1)}{
\Gamma(\beta+1-m-\alpha)}T^{-(m+\alpha)}(1-t/T)^{\beta-\alpha-m}.
\end{equation}
\end{lemma}
\begin{lemma}\label{lemma2}
Let $T>0$, $0<\alpha<1$, $m\geq0$ and $p>1$. Then, we have
\begin{equation}
\int_0^T (w(t))^{-\frac{1}{p-1}}|D_{t|T}^{m+\alpha}
w(t)|^{\frac{p}{p-1}}\,dt=C\,T^{1-(m+\alpha)\frac{p}{p-1}}.
\end{equation}
\end{lemma}
\begin{proof}  Using Lemma \ref{lemma1} we have 
\begin{eqnarray*}
\int_{0}^T (w(t))^{-\frac{1}{p-1}}|D_{t|T}^{m+\alpha}
w(t)|^{\frac{p}{p-1}}\,dt&=&C\,T^{-(m+\alpha)\frac{p}{p-1}}\int_{0}^T (w(t))^{-\frac{1}{p-1}}(w(t))^{\frac{p({\beta}-\alpha-m)}{(p-1)\beta}}\,dt\\
&=&C\,T^{-(m+\alpha)\frac{p}{p-1}}\int_{0}^T(1-t/T)^{\beta-(m+\alpha)\frac{ p}{p-1}}\,dt\\
&=&C\,T^{1-(m+\alpha)\frac{p}{p-1}}\int_{0}^1(1-s)^{\beta-(m+\alpha)\frac{ p}{p-1}}\,ds\\
&=&C\,T^{1-(m+\alpha)\frac{p}{p-1}},
\end{eqnarray*}
where we notice that since $\beta \gg1$ is big enough, it guarantees the integrability of the last integral.
\end{proof}

\bdn[\cite{Kwanicki,Silvestre}]
\fontshape{n}
\selectfont
Let $s \in (0,1)$. Let $X$ be a suitable set of functions defined on $\R^n$. Then, the fractional Laplacian $(-\Delta)^s$ in $\R^n$ is a non-local operator given by
$$ (-\Delta)^s: \,\,v \in X  \to (-\Delta)^s v(x):= C_{n,s}\,\, p.v.\int_{\R^n}\frac{v(x)- v(y)}{|x-y|^{n+2s}}dy, $$
as long as the right-hand exists, where $p.v.$ stands for Cauchy's principal value, $C_{n,s}:= \frac{4^s \Gamma(\frac{n}{2}+s)}{\pi^{\frac{n}{2}}\Gamma(-s)}$ is a normalization constant and $\Gamma$ denotes the Gamma function.
\label{def1}
\edn

\begin{lemma}\label{lemma4}
Let $\langle x\rangle:=(1+(|x|-1)^4)^{1/4}$ for all $x\in\mathbb{R}^n$. Let $s \in (0,1]$ and $\phi$ be a function defined by
\begin{equation}\label{testfunction}
\phi(x)=\left\{
\begin{array}{ll}
1&\hbox{if}\;\;|x|\leq1,\\
{}\\
\langle x\rangle^{-n-2s}&\hbox{if}\;\;|x|\geq1.
\end{array}
\right.
\end{equation}
Then, $\phi\in C^2(\mathbb{R}^n)$ and the following estimate holds:
\begin{equation}\label{10}
\left|(-\Delta)^s\phi(x)\right|\lesssim \phi(x) \quad\hbox{for all}\;x\in\mathbb{R}^n.
\end{equation}
\end{lemma}
\begin{proof}
Let us denote $r:=|x|$. As $\phi$ is a radial function, we have 
$$\nabla\phi(x)=\frac{x}{r}\phi'(r)=\left\{
\begin{array}{ll}
0&\hbox{if}\;\;r\leq1,\\
{}\\
-(n+2s)\frac{x}{r}(r-1)^3\langle x\rangle^{-n-2s-4}&\hbox{if}\;\;r\geq1,
\end{array}
\right.
$$
and
\begin{equation}\label{delta}
\Delta\phi(x)=\phi''(r)+\frac{n-1}{r}\phi'(r)=\left\{
\begin{array}{ll}
0&\hbox{if}\;\;r\leq1,\\
{}\\
-3(n+2s)(r-1)^2\langle x\rangle^{-n-2s-4}\\
\quad +(n+2s)(n+2s+4)(r-1)^6\langle x\rangle^{-n-2s-8}\\
\quad -(n+2s)\frac{n-1}{r}(r-1)^3\langle x\rangle^{-n-2s-4}&\hbox{if}\;\;r\geq1.
\end{array}
\right.
\end{equation}
We can check easily that $\phi\in C^2(\mathbb{R}^n)$. Moreover, $\|\phi\|_{L^\infty(\mathbb{R}^n)},\|\partial_x^2\phi\|_{L^\infty(\mathbb{R}^n)}\leq C$ which allows us to remove the principal value of the integral at the origin and conclude that
$$(-\Delta)^s \phi(x)= -\frac{C_{n,s}}{2} \int_{\mathbb{R}^n}\frac{\phi(x+y)+ \phi(x-y)- 2\phi(x)}{|y|^{n+ 2s}}\,dy.$$
To prove the desired estimate, we have to distinguish two cases.\\

\noindent {\bf Case 1: $|x|\leq 2$.} We divide the above integral into two parts as follows:
\begin{eqnarray}\label{E1}
|(-\Delta)^s \phi(x)|&\lesssim & \int_{|y|\leq1}\frac{|\phi(x+y)+ \phi(x-y)- 2\phi(x)|}{|y|^{n+ 2s}}\,dy+ \int_{|y|\geq1}\frac{|\phi(x+y)+ \phi(x-y)- 2\phi(x)|}{|y|^{n+ 2s}}\,dy\nonumber\\
&\lesssim & \|\partial_x^2\phi\|_{L^\infty(\mathbb{R}^n)}\int_{|y|\leq1}\frac{1}{|y|^{n+ 2s-2}}dy+ \|\phi\|_{L^\infty(\mathbb{R}^n)}\int_{|y|\geq1}\frac{1}{|y|^{n+ 2s}}\,dy\nonumber\\
&\lesssim & 1,
\end{eqnarray}
where we have used the boundedness of the above two integrals.\\

\noindent {\bf Case 2: $|x|\geq 2$.} In this case, we re-write the fractional laplacian as follows:
\begin{eqnarray}\label{E2}
(-\Delta)^s \phi(x)&= & -\frac{C_{n,s}}{2} \int_{|y|\geq 2|x|}\frac{\phi(x+y)+ \phi(x-y)- 2\phi(x)}{|y|^{n+ 2s}}\,dy\nonumber\\
&{}&-\frac{C_{n,s}}{2} \int_{\frac{1}{2}|x|\leq |y|\leq 2|x|}\frac{\phi(x+y)+ \phi(x-y)- 2\phi(x)}{|y|^{n+ 2s}}\,dy\nonumber\\
&{} &-\frac{C_{n,s}}{2} \int_{|y|\leq\frac{1}{2}|x|}\frac{\phi(x+y)+ \phi(x-y)- 2\phi(x)}{|y|^{n+ 2s}}\,dy\nonumber\\
&=:& I_1+I_2+I_3.
\end{eqnarray}
We start to estimate $I_1$. We notice that when $|y|\geq 2|x|$, we have $|x\pm y|\geq |y|-|x|\geq|x|\geq 2$. Using the monotonicity of $\phi$, we obtain $\phi(x\pm y)\leq \phi(x)=\langle x\rangle^{-n-2s}$. Therefore, it holds
\begin{equation}\label{I1}
|I_1|\lesssim 4\phi(x) \int_{|y|\geq 2|x|}\frac{1}{|y|^{n+ 2s}}dy\lesssim \langle x\rangle^{-n-2s} \int_{|y|\geq 2|x|}\frac{1}{|y|^{1+ 2s}}d|y|\lesssim \langle x\rangle^{-n-2s}|x|^{-2s}\lesssim \langle x\rangle^{-n-4s},
\end{equation}
where we have used the fact that $\langle x\rangle\approx |x|-1\approx |x|$ when $|x|\geq 2$.\\
For the estimation of $I_2$, it is clear that $|y|\approx |x|$ inside the integral and
$$\left\{y\in\mathbb{R}^n:\;\;\frac{1}{2}|x|\leq |y|\leq 2|x|\right\}\subseteq \big\{y\in\mathbb{R}^n:\;\;|x\pm y|\leq 3|x|\big\}.$$
Thus, it follows that
\begin{eqnarray}\label{I2}
|I_2|&\lesssim & |x|^{-n-2s}\left( \int_{|x+y|\leq 3|x|}\phi(x+y)\,dy+ \int_{|x-y|\leq 3|x|}\phi(x-y)\,dy+\phi(x) \int_{\frac{1}{2}|x|\leq |y|\leq 2|x|}1\,dy\right)\nonumber\\
&\lesssim & |x|^{-n-2s}\left( \int_{|x+y|\leq 3|x|}\phi(x+y)\,dy+  \langle x\rangle^{-n-2s}|x|^n\right)\nonumber\\
&\lesssim & |x|^{-n-2s}\left(1+  \langle x\rangle^{-2s}\right)\nonumber\\
&\lesssim & |x|^{-n-2s},
\end{eqnarray}
where we have used the following relation: 
\begin{eqnarray*}
\int_{|x-y|\leq 3|x|}\phi(x-y)\,dy
&=& \int_{|x+y|\leq 3|x|}\phi(x+y)\,dy\\
&=& \int_{2\leq |x+y|\leq 3|x|}\phi(x+y)\,dy+ \int_{|x+y|\leq 2}\phi(x+y)\,dy\\
&\lesssim&  \int_{2}^{3|x|}(1+(r-1)^4)^{-(n+2s)/4}r^{n-1}\,dr+   \int_{0}^{2}r^{n-1}\,dr\\
&\lesssim&  \int_{2}^{3|x|}r^{-2s-1}\,dr+   \int_{0}^{2}r^{n-1}\,dr\\
&\lesssim& 1.
\end{eqnarray*}
We arrive to estimate the third integral $I_3$. Using the second order Taylor expansion for $\phi$, we obtain
\begin{eqnarray}\label{I_3}
|I_3|&\lesssim & \int_{|y|\leq\frac{1}{2}|x|}\frac{|\phi(x+y)+ \phi(x-y)- 2\phi(x)|}{|y|^{n+ 2s}}\,dy\nonumber\\
&\lesssim & \int_{|y|\leq\frac{1}{2}|x|}\,\max_{\theta\in[0,1]}|\partial^2_x\phi(x\pm\theta y)|\frac{1}{|y|^{n+ 2s-2}}\,dy.
\end{eqnarray}
On the other hand, we have the following estimate for $\theta\in[0,1]$ and $r=|x\pm\theta y|$:
\begin{eqnarray*}
|\partial^2_x\phi(x\pm\theta y)|&\leq&\left\{
\begin{array}{ll}
0&\hbox{if}\;\;r\leq1,\\
{}\\
C(r-1)^2(1+(r-1)^4)^{-(n+2s+4)/4}\\
\quad +C(r-1)^6(1+(r-1)^4)^{-(n+2s+8)/4}\\
\quad +C\frac{(r-1)^3}{r}(1+(r-1)^4)^{-(n+2s+4)/4}&\hbox{if}\;\;r\geq1,
\end{array}
\right.\\
&\leq&\left\{
\begin{array}{ll}
0&\hbox{if}\;\;r\leq1,\\
{}\\
C(1+(r-1)^4)^{-(n+2s+2)/4}&\hbox{if}\;\;r\geq1.\\
\end{array}
\right.\\
\end{eqnarray*}
As 
$$|x\pm\theta y|\geq|x|-\theta|y|\geq |x|-|x|/2=|x|/2\geq 1,$$ 
we deduce
$$|\partial^2_x\phi(x\pm\theta y)|\leq C  \langle x\pm\theta y\rangle^{-n-2s-2}.$$
If $|x\pm\theta y|\geq 2$, then $ \langle x\pm\theta y\rangle\approx |x\pm\theta y|\geq |x|/2\gtrsim  \langle x\rangle$, which implies 
$$ \langle x\pm\theta y\rangle^{-n-2s-2}\lesssim  \langle x\rangle^{-n-2s-2}.$$
If $|x\pm\theta y|\leq 2$, then $|x|/2\leq |x\pm\theta y|\leq 2$, which implies $2\leq|x|\leq 4$. Therefore, we may arrive at $ \{\langle x\pm\theta y\rangle\approx \langle x\rangle \approx 1$ and
$$ \langle x\pm\theta y\rangle^{-n-2s-2}\approx  \langle x\rangle^{-n-2s-2}.$$
This yields
$$|\partial^2_x\phi(x\pm\theta y)|\leq C  \langle x\rangle^{-n-2s-2} \quad\hbox{for all}\;\theta\in[0,1].$$
By (\ref{I_3}), we conclude that
\begin{eqnarray}\label{I3}
|I_3|&\lesssim & \langle x\rangle^{-n-2s-2} \int_{|y|\leq\frac{1}{2}|x|}\frac{1}{|y|^{n+ 2s-2}}\,dy\nonumber\\
&\lesssim & \langle x\rangle^{-n-2s-2} \int_{0}^{|x|/2} |y|^{-2s+1}\,d|y|\nonumber\\
&\lesssim & \langle x\rangle^{-n-2s-2} |x|^{-2s+2}\nonumber\\
&\lesssim & \langle x\rangle^{-n-4s} .
\end{eqnarray}
Combining (\ref{I1}), (\ref{I2}) and (\ref{I3}) we conclude from (\ref{E2}) that
\begin{eqnarray}\label{E3}
|(-\Delta)^s \phi(x)|\lesssim \langle x\rangle^{-n-2s} \quad\hbox{ for all }\;|x|\geq 2.
\end{eqnarray}
Finally, both (\ref{E1}) and (\ref{E3}) imply
$$|(-\Delta)^s \phi(x)|\lesssim \langle x\rangle^{-n-2s}\le \phi(x) \quad\hbox{for all}\;x\in\mathbb{R}^n,$$
which we wanted to prove.
\end{proof}

\begin{lemma}\label{lemma7}\cite[Lemma~2.4]{DaoReissig}
Let $s \in (0,1]$. Let $\psi$ be a smooth function satisfying $\partial_x^2\psi\in L^\infty(\mathbb{R}^n)$. For any $R>0$, let $\psi_R$ be a function defined by
$$ \psi_R(x):= \psi(x/R) \quad \text{ for all } x \in \R^n.$$
Then, $(-\Delta)^s \psi_R$ satisfies the following scaling properties:
$$(-\Delta)^s \psi_R(x)= R^{-2s}((-\Delta)^s\psi)(x/R) \quad \text{ for all } x \in \R^n. $$
\end{lemma}

\begin{lemma}\label{lemma6} Let $s \in (0,1]$. Let $R>0$ and $p>1$. Then, the following estimate holds
$$\int_{\mathbb{R}^n}(\phi_R(x))^{-\frac{1}{p-1}}\,\big|(-\Delta)^s\phi_R(x)\big|^{\frac{p}{p-1}}\, dx\lesssim R^{-\frac{2sp}{p-1}+n},$$
where $\phi_R(x):= \phi({x}/{R})$ and $\phi$ is given in (\ref{testfunction}).
\end{lemma}
\begin{proof} If $0<s<1$, then using the change of variables $\tilde{x}=x/R$ and Lemma \ref{lemma7} we have $(-\Delta)^s\phi_R(x)=R^{-2s}(-\Delta)^s\phi(\tilde{x})$. Therefore, by Lemma \ref{lemma4} we conclude that
$$\int_{\mathbb{R}^n}(\phi_R(x))^{-\frac{1}{p-1}}\,\big|(-\Delta)^s \phi_R(x)\big|^{\frac{p}{p-1}}\, dx\lesssim R^{-\frac{2sp}{p-1}+n}\int_{\mathbb{R}^n}\phi(\tilde{x})\, d \tilde{x}\lesssim R^{-\frac{2sp}{p-1}+n}.$$
For the case of $s=1$, we may repeat the same calculation as (\ref{delta}) in Lemma \ref{lemma4} to conclude
$$\big|\Delta \phi(\tilde{x})\big|\lesssim  \langle \tilde{x}\rangle^{-n-2s-2},$$
which follows immediately the desired estimate by the change of variables.
\end{proof}


\section{Proof of main result} \label{Sec.Proof}
Before starting our proof, we define weak solutions for (\ref{1}).
\begin{dn}
\textcolor{red}{Let $p>1$.} Let $T>0$ and $(u_0,u_1)\in H^\sigma(\mathbb{R}^n)\times L^2(\mathbb{R}^n)$. A function $u$ is said to be a weak solution to \eqref{1} on $[0,T)$, if 
$$ u\in L^p((0,T),L^{2p}(\mathbb{R}^n))\cap  L^1((0,T),L^2(\mathbb{R}^n))
$$
 and the following formuation
 	\begin{align*}
&\Gamma(\alpha)\int_0^T\int_{\mathbb{R}^n}I^{\alpha}_{0|t}(|u(t,x)|^p)\varphi(t,x)\,dx\,dt\\
&\quad +\int_{\mathbb{R}^n}u_1(x)\varphi(0,x)\,dx-\int_{\mathbb{R}^n}u_0(x)\varphi_t(0,x)\,dx+\mu \int_{\mathbb{R}^n}(-\Delta)^{\sigma/2}u_0(x)\varphi(0,x)\,dx \\
&\quad =\int_0^T\int_{\mathbb{R}^n}u(t,x)\varphi_{tt}(t,x)\,dx\,dt- \int_0^T\int_{\mathbb{R}^n}u(t,x)\,\Delta\varphi(t,x) \,dx\,dt \\
&\qquad - \mu\int_0^T\int_{\mathbb{R}^n}u(t,x)\,(-\Delta)^{\sigma/2}\varphi_t(t,x)\,dx\,dt
\end{align*}
	holds for any test function$\varphi\in C([0,T];H^2(\mathbb{R}^n)) \cap C^1([0,T];H^\sigma(\mathbb{R}^n)) \cap C^2([0,T];L^2(\mathbb{R}^n))$ such that $\varphi(T,x)=0$ and $\varphi_t(T,x)=0$ for all $x\in \R^n$.
\par We say that $u$ is a global weak solution to \eqref{1} if $u$ is a weak solution to \eqref{1} on $[0,T)$ for any $T>0$.
\end{dn}
\begin{proof}[Proof of Theorem \ref{blow-up}]
The proof is based on a contradiction. Suppose that $u$ is a global weak solution to (\ref{1}), then $u$ satisfies
\begin{align}
&\Gamma(\alpha)\int_0^T\int_{\mathbb{R}^n}I^{\alpha}_{0|t}(|u(t,x)|^p)\varphi(t,x)\,dx\,dt  \nonumber \\
&\quad +\int_{\mathbb{R}^n}u_1(x)\varphi(0,x)\,dx-\int_{\mathbb{R}^n}u_0(x)\varphi_t(0,x)\,dx+\mu \int_{\mathbb{R}^n}(-\Delta)^{\sigma/2}u_0(x)\varphi(0,x)\,dx \nonumber \\
&\quad =\int_0^T\int_{\mathbb{R}^n}u(t,x)\varphi_{tt}(t,x)\,dx\,dt- \int_0^T\int_{\mathbb{R}^n}u(t,x)\,\Delta\varphi(t,x) \,dx\,dt  \nonumber \\
&\qquad  - \mu\int_0^T\int_{\mathbb{R}^n}u(t,x)\,(-\Delta)^{\sigma/2}\varphi_t(t,x)\,dx\,dt \label{3}
\end{align}
where $\alpha:=1-\gamma\in(0,1)$, for all test function $\varphi$ such that $\varphi(T,\cdotp)=\varphi_t(T,\cdotp)=0$ for all $T\gg1$. \\
Let $R$ and $T$ be large parameters in $(0,\ity)$. We define the following auxiliary functions:
$$ \phi_R(x):=\phi(x/R) \quad \text{ and }\quad \tilde{\varphi}(t,x):= \phi_R(x)\,w(t), $$
where $\phi(x)$ and $w(t)$ are defined in Section \ref{Sec.Pre} with $s=\sigma/2$. Then, we define the test function
$$\varphi(t,x):=D^\alpha_{t|T}\left(\tilde{\varphi}(t,x)\right)= \phi_R(x) D^{\alpha}_{t|T}\left(w(t)\right).$$
From (\ref{3}), using (\ref{5}) and (\ref{6}) we have
\begin{align}
&\Gamma(\alpha)\int_0^T\int_{\mathbb{R}^n}I^{\alpha}_{0|t}(|u(t,x)|^p)D^\alpha_{t|T}\left(\tilde{\varphi}(t,x)\right)\,dx\,dt\nonumber\\
&\quad +C\,T^{-\alpha}\int_{\mathbb{R}^n}\big(u_1(x)+ \mu(-\Delta)^{\sigma/2}u_0(x)\big)\phi_R(x) \,dx-C\,T^{-1-\alpha}\int_{\mathbb{R}^n}u_0(x)\phi_R(x)\,dx\nonumber\\
&\quad =\int_0^T\int_{\mathbb{R}^n}u(t,x)\,\phi_R(x)\,D^{2+\alpha}_{t|T}\left(w(t)\right)\,dx\,dt- \int_0^T\int_{\mathbb{R}^n}u(t,x)D^{\alpha}_{t|T}\left(w(t)\right)\Delta \phi_R(x)\,dx\,dt\nonumber\\
&\qquad +\mu\int_0^T\int_{\mathbb{R}^n}u(t,x)D^{1+\alpha}_{t|T}\left(w(t)\right)(-\Delta)^{\sigma/2}(\phi_R(x))\,dx\,dt\label{4}.
\end{align}
 Using (\ref{IP}) and then (\ref{7}) we arrive at 
 \begin{align}
&I_R+C\,T^{-\alpha}\int_{\mathbb{R}^n}\big(u_1(x)+\mu(-\Delta)^{\sigma/2}u_0(x)\big)\phi_R(x) \,dx-C\,T^{-1-\alpha}\int_{\mathbb{R}^n}u_0(x)\phi_R(x)\,dx\nonumber\\ 
&\quad =C\int_0^T\int_{\mathbb{R}^n}u(t,x)\,\phi_R(x)\, D^{2+\alpha}_{t|T}\left(w(t)\right)\,dx\,dt - C\int_0^T\int_{|x|\geq R}u(t,x)D^{\alpha}_{t|T}\left(w(t)\right)\Delta \phi_R(x)\,dx\,dt\nonumber\\
&\qquad +C\int_0^T\int_{\mathbb{R}^n}u(t,x)D^{1+\alpha}_{t|T}\left(w(t)\right)(-\Delta)^{\sigma/2}(\phi_R(x))\,dx\,dt\nonumber\\
&\quad =:J_1+J_2+J_3\label{8},
\end{align}
where
$$I_R:=\int_0^T\int_{\mathbb{R}^n}|u(t,x)|^p\,\tilde{\varphi}(t,x)\,dx\,dt.$$
After applying H\"{o}lder's inequality with $\frac{1}{p}+\frac{1}{p'}=1$, where $p'$ is the conjugate of $p$, we can proceed the estimate for $J_1$ as follows:
\begin{align*}
|J_1| &\le C\,\int_0^T\int_{\mathbb{R}^n}|u(t,x)|\, \phi_R(x)\,\big|D^{2+\alpha}_{t|T}\left(w(t)\right)\big|\,dx\,dt \\
&=C\,\int_0^T\int_{\mathbb{R}^n}|u(t,x)|\, (\tilde{\varphi}(t,x))^{\frac{1}{p}}(\tilde{\varphi}(t,x))^{-\frac{1}{p}}\,\varphi_R(x)\,\big|D^{2+\alpha}_{t|T}\left(w(t)\right)\big|\, dx\,dt \\
&\lesssim I_{R}^{\frac{1}{p}} \Big(\int_0^T\int_{\mathbb{R}^n}\phi_R(x)\,(w(t))^{-\frac{p'}{p}} \,\big|D^{2+\alpha}_{t|T}\left(w(t)\right)\big|^{p'}\, dx\,dt\Big)^{\frac{1}{p'}}.
\end{align*}
By the change of variables $\tilde{t}:= t/T$ and $\tilde{x}:= x/R$, using Lemma \ref{lemma2} we get
\begin{equation} \label{10}
|J_1| \lesssim I_{R}^{\frac{1}{p}}\,R^{\frac{n}{p'}}\,T^{\frac{1}{p'}-2-\alpha} \Big(\int_{\mathbb{R}^n} \langle \tilde{x}\rangle^{-n-\sigma}\,d\tilde{x}\Big)^{\frac{1}{p'}}\lesssim I_{R}^{\frac{1}{p}}\,R^{\frac{n}{p'}}\,T^{\frac{1}{p'}-2-\alpha}.
\end{equation}
Similarly, applying H\"{o}lder's inequality again and  using Lemma \ref{lemma2} and Lemma \ref{lemma6} (with $s=1$), we have
\begin{align}\label{11}
|J_2|& \lesssim \tilde{I}_{R}^{\frac{1}{p}} \Big(\int_0^T\int_{|x|\geq R}(\phi_R(x))^{-\frac{p'}{p}}\,(w(t))^{-\frac{p'}{p}} \,\big|D^{\alpha}_{t|T}\left(w(t)\right)\big|^{p'}\,\big|\Delta \phi_R(x)\big|^{p'}\, dx\,dt\Big)^{\frac{1}{p'}}\nonumber\\
& \lesssim \tilde{I}_{R}^{\frac{1}{p}} \Big(\int_0^T\,(w(t))^{-\frac{p'}{p}} \,\big|D^{\alpha}_{t|T}\left(w(t)\right)\big|^{p'}\,dt\Big)^{\frac{1}{p'}}\Big(\int_{|x|\geq R}(\phi_R(x))^{-\frac{p'}{p}}\,\big|\Delta \phi_R(x)\big|^{p'}\, dx\Big)^{\frac{1}{p'}}\nonumber\\
& \lesssim \tilde{I}_{R}^{\frac{1}{p}}\,T^{\frac{1}{p'}-\alpha}\Big(\int_{\R^n}(\phi_R(x))^{-\frac{p'}{p}}\,\big|\Delta \phi_R(x)\big|^{p'}\, dx\Big)^{\frac{1}{p'}} \nonumber \\
& \lesssim \tilde{I}_{R}^{\frac{1}{p}}\,T^{\frac{1}{p'}-\alpha}R^{\frac{n}{p'}-2},
\end{align}
where 
$$\tilde{I}_R:=\int_0^T\int_{|x|\geq R}|u(t,x)|^p\,\tilde{\varphi}(t,x)\,dx\,dt.$$
Moreover, Lemma \ref{lemma2} and Lemma \ref{lemma6} with $s= \sigma/2$ imply
\begin{align}\label{12}
|J_3|& \lesssim I_{R}^{\frac{1}{p}} \Big(\int_0^T\int_{\mathbb{R}^n}(\phi_R(x))^{-\frac{p'}{p}}\,(w(t))^{-\frac{p'}{p}} \,\big|D^{1+\alpha}_{t|T}\left(w(t)\right)\big|^{p'}\,\big|(-\Delta)^{\sigma/2} \phi_R(x)\big|^{p'}\, dx\,dt\Big)^{\frac{1}{p'}}\nonumber\\
& \lesssim I_{R}^{\frac{1}{p}} \Big(\int_0^T\,(w(t))^{-\frac{p'}{p}} \,\big|D^{1+\alpha}_{t|T}\left(w(t)\right)\big|^{p'}\,dt\Big)^{\frac{1}{p'}}\Big(\int_{\mathbb{R}^n}(\phi_R(x))^{-\frac{p'}{p}}\,\big|(-\Delta)^{\sigma/2} \phi_R(x)\big|^{p'}\, dx\Big)^{\frac{1}{p'}}\nonumber\\
& \lesssim I_{R}^{\frac{1}{p}}\,T^{\frac{1}{p'}-\alpha-1}R^{\frac{n}{p'}-\sigma}.
\end{align}
Using the estimates from (\ref{10}) to (\ref{12}) into (\ref{8}) we arrive at
\begin{eqnarray*}
&{}&I_R+C\,T^{-\alpha}\int_{\mathbb{R}^n}\big(u_1(x)+ \mu(-\Delta)^{\sigma/2}u_0(x)\big)\phi_R(x) \,dx-C\,T^{-1-\alpha}\int_{\mathbb{R}^n}u_0(x)\phi_R(x)\,dx\\
&{}&\qquad \lesssim I_{R}^{\frac{1}{p}}\left(R^{\frac{n}{p'}}\,T^{\frac{1}{p'}-2-\alpha}+T^{\frac{1}{p'}-\alpha-1}R^{\frac{n}{p'}-\sigma}\right)+\tilde{I}_{R}^{\frac{1}{p}}\,T^{\frac{1}{p'}-\alpha}R^{\frac{n}{p'}-2},
\end{eqnarray*}
that is,
\begin{eqnarray}\label{9}
&{}&I_R+C\,T^{-\alpha}\int_{\mathbb{R}^n}\big(u_1(x)+\mu(-\Delta)^{\sigma/2}u_0(x)\big)\phi_R(x) \,dx\nonumber\\
&{}&\qquad \lesssim I_{R}^{\frac{1}{p}}R^{\frac{n}{p'}}\,T^{\frac{1}{p'}-\alpha}\left(T^{-2}+ T^{-1}R^{-\sigma}\right)+\tilde{I}_{R}^{\frac{1}{p}}\,T^{\frac{1}{p'}-\alpha}R^{\frac{n}{p'}-2}+ C\,T^{-1-\alpha}\int_{\mathbb{R}^n}u_0(x)\phi_R(x)\,dx.\qquad
\end{eqnarray}
Due of (\ref{optimal9.1}) and the fact that $\phi_R\rightarrow 1$ as $R\rightarrow\infty$, there exists a sufficiently large constant $R_0> 0$ such that
\begin{equation}\label{13}
\int_{\R^n} \big(u_1(x)+ \mu(-\Delta)^{\sigma/2} u_0(x)\big)\phi_R(x)\, dx >0,
\end{equation}
for all $R > R_0$. Our considerations are splitted into the following several cases. \\

\noindent {\bf Case 1}: If $p<p_c$, then we take $R=T^{\frac{1}{2-\tilde{\sigma}}}$. Both (\ref{9}) and (\ref{13}) imply
$$
I_R\lesssim I_{R}^{\frac{1}{p}} T^{\frac{1}{p'}-\alpha+\frac{n}{(2-\tilde{\sigma})p'}-\frac{2}{2-\tilde{\sigma}}} +T^{-1-\alpha}\int_{\mathbb{R}^n}u_0(x)\phi_R(x)\,dx,
$$
where we have used the fact that $\tilde{I}_{R}\leq I_{R}$ and $\tilde{\sigma}\leq1$. Using the following Young's inequality: 
$$ab\leq\frac{1}{p}a^p+\frac{1}{p'}b^{p'}\quad \text{ for all }a,b>0,$$ 
we conclude that
\begin{equation}
\frac{1}{p'}I_R\lesssim \frac{1}{p'}T^{1-\alpha p'+\frac{n}{2-\tilde{\sigma}}-\frac{2p'}{2-\tilde{\sigma}}} +T^{-1-\alpha}\int_{\mathbb{R}^n}u_0(x)\phi_R(x)\,dx. \label{132}
\end{equation}
It is clear that the condition $p<p_c$ follows immediately $1-\alpha p'+\frac{n}{2-\tilde{\sigma}}-\frac{2p'}{2-\tilde{\sigma}}<0$. Therefore, letting $T\rightarrow\infty$ we infer that $u=0 $ a.e.. This implies, by using (\ref{9}) again (or directly from (\ref{4})), that 
$$\int_{\mathbb{R}^n}\big(u_1(x)+ \mu(-\Delta)^{\sigma/2}u_0(x)\big)\phi_R(x) \,dx \lesssim T^{-1}\int_{\mathbb{R}^n}u_0(x)\phi_R(x)\,dx
$$
for all $R,\,T\geq1$. Hence, letting $T\rightarrow\infty$ we obtain a contradiction to (\ref{optimal9.1}).\\

\noindent {\bf Case 2}: If $p=p_c$, that is, $1-\alpha p'+\frac{n}{2-\tilde{\sigma}}-\frac{2p'}{2-\tilde{\sigma}}=0$, then we can see the following estimate from (\ref{132}):
$$
\frac{1}{p'}I_R\lesssim \frac{1}{p'}+T^{-1-\alpha}\int_{\mathbb{R}^n}u_0(x)\phi_R(x)\,dx.
$$
Thus, it follows that $I_R \le C$ as $T \to \ity$. Using Beppo Levi's theorem on monotone convergence, we derive
$$\int_0^\infty\int_{\mathbb{R}^n}|u(t,x)|^p\,dx\,dt=\lim_{T\rightarrow\infty}\int_0^T\int_{\mathbb{R}^n}|u(t,x)|^p\,\tilde{\varphi}(t,x)\,dx\,dt= \lim_{T\rightarrow\infty}I_R \leq C, $$
that is, $u\in L^p((0,\infty)\times\mathbb{R}^n)$. On the other hand, taking $R=T^{\frac{1}{2-\tilde{\sigma}}}K^{-\frac{1}{2-\tilde{\sigma}}}$ with some constants $K\ge 1$ into (\ref{9}) we obtain
\begin{align*}
I_R &\lesssim I_{R}^{\frac{1}{p}} \Big(T^{-\frac{2(1-\tilde{\sigma})}{2-\tilde{\sigma}}}K^{-\frac{n}{(2-\tilde{\sigma})p'}}+ T^{-\frac{\sigma-\tilde{\sigma}}{2-\tilde{\sigma}}}K^{-\frac{n}{(2-\tilde{\sigma})p'}+\frac{\sigma}{2-\tilde{\sigma}}}\Big)+ \tilde{I}_{R}^{\frac{1}{p}}\,\,K^{-\frac{n}{(2-\tilde{\sigma})p'}+\frac{2}{2-\tilde{\sigma}}}\\ 
&\quad +T^{-1-\alpha}\int_{\mathbb{R}^n}u_0(x)\phi_R(x)\,dx,
\end{align*}
where we have used the fact that $p=p_c$. Using the following Young's inequality: 
$$ab\leq\frac{1}{p}a^p+\frac{1}{p'}b^{p'}\quad \text{ for all }a,b>0,$$ 
we conclude that
\begin{align}
\frac{1}{p'}I_R &\lesssim \frac{1}{p'}\Big(T^{-\frac{2(1-\tilde{\sigma})p'}{2-\tilde{\sigma}}}K^{-\frac{n}{2-\tilde{\sigma}}}+ T^{-\frac{(\sigma-\tilde{\sigma})p'}{2-\tilde{\sigma}}}K^{-\frac{n-\sigma p'}{2-\tilde{\sigma}}}\Big)+ \tilde{I}_{R}^{\frac{1}{p}}\,\,K^{-\frac{n}{(2-\tilde{\sigma})p'}+\frac{2}{2-\tilde{\sigma}}} \nonumber \\ 
& \quad +T^{-1-\alpha}\int_{\mathbb{R}^n}u_0(x)\phi_R(x)\,dx. \label{14}
\end{align}
If $\sigma\in(0,1]$, using the fact that $u\in L^p((0,\infty)\times\mathbb{R}^n)$, it holds
$$\lim_{T\rightarrow\infty} \tilde{I}_{R}=\lim_{T\rightarrow\infty} \int_0^T\int_{|x|\geq T^{\frac{1}{2-\sigma}}K^{-\frac{1}{2-\sigma}}}|u(t,x)|^p\,\tilde{\varphi}(t,x)\,dx\,dt=0. $$
Consequently, from (\ref{14}) we may arrive at
$$
\int_0^\infty\int_{\mathbb{R}^n}|u(t,x)|^p\,dx\,dt\lesssim K^{-\frac{n}{2-\sigma}}+K^{-\frac{n-\sigma p'}{2-\sigma}} \quad \text{ as } T \to \ity.$$
Finally, using again $p=p_c$ and $\tilde{\sigma}=\sigma$, we can check easily that $n>\sigma p'$. Therefore, taking $K$ big enough we obtain the desired result similarly to Case 1.\\
If $\sigma\in(1,2)$, using the fact that $u\in L^p((0,\infty)\times\mathbb{R}^n)$, it holds
$$\lim_{T\rightarrow\infty} \tilde{I}_{R}=\lim_{T\rightarrow\infty} \int_0^T\int_{|x|\geq TK^{-1}}|u(t,x)|^p\,\tilde{\varphi}(t,x)\,dx\,dt=0. $$
Consequently, from (\ref{14}) we may arrive at
$$
\int_0^\infty\int_{\mathbb{R}^n}|u(t,x)|^p\,dx\,dt\lesssim K^{-n} \quad \text{ as } T \to \ity,$$
where we have used that $\sigma-\tilde{\sigma}=\sigma-1>0$. Therefore, taking $K$ big enough we obtain the desired result similarly to Case 1.\\

\noindent {\bf Case 3}: If $p<\gamma^{-1}$, then this means that $\frac{1}{p'}-\alpha<0$. In this case, we suppose $R<T$ such that $T$ and $R$ cannot go to infinity simultaneously (or we can choose $R=\ln T$ as in \cite{Dabbicco2}). Therefore, by letting $T\rightarrow\infty$ and then $R$ big enough in (\ref{9}) we obtain the desired contradiction.
\end{proof}


\begin{thebibliography}{ABC}
\bibitem{BonforteVazquez} M. Bonforte, J.L. V\'{a}zquez, \textit{Quantitative local and global a priori estimates for fractional nonlinear diffusion equations}, Adv. Math., 250 (2014), 242-284.

\bibitem{Dabbicco1} M. D'Abbicco, \textit{The influence of a nonlinear memory on the damped wave equation}, Nonlinear Anal., 95 (2014), 130-145

\bibitem{Dabbicco2} M. D'Abbicco, \textit{A wave equation with structural damping and nonlinear memory}, NoDEA, 21 (2014), 751-773.

\bibitem{DabbiccoReissig} M. D'Abbicco, M. Reissig, \textit{Semilinear structural damped waves}, Math. Methods Appl. Sci., 37 (2014), 1570-1592.

\bibitem{DaoReissig} T.A. Dao, M. Reissig, \textit{A blow-up result for semi-linear structurally damped $\sigma$-evolution equations}, preprint on arXiv:1909.01181v1, 2019.

\bibitem{Fino2011} A. Fino, \textit{Critical exponent for damped wave equations with nonlinear memory}, Nonlinear Anal., 74 (2011), 5495-5505.

\bibitem{Finokirane} A. Z. Fino, M. Kirane, \textit{Qualitative properties of solutions to a time-space fractional evolution equation}, Quart Appl. Math. 70(1) (2012),133-157.

\bibitem{Ju} N. Ju, \textit{The Maximum Principle and the Global Attractor
for the Dissipative 2-D Quasi-Geostrophic Equations}, Comm. Pure.
Appl. Ana. (2005), 161-181.

\bibitem{KSTr}A. A. Kilbas, H. M. Srivastava, J. J. Trujillo, \textit{Theory and Applications of Fractional Differential Equations}, 2006.

\bibitem{Kwanicki} M. Kwa\'{s}nicki, \textit{Ten equivalent definitions of the fractional laplace operator}, Fract. Calc. Appl. Anal., 20 (2017), 7-51.

\bibitem{SKM} S. G. Samko, A. A. Kilbas, O. I. Marichev, \textit{Fractional integrals and derivatives}, Theory and Applications, Gordon and
Breach Science Publishers, 1987.

\bibitem{Silvestre} L. Silvestre, \textit{Regularity of the obstacle problem for a fractional power of the Laplace operator}, Comm. Pure Appl. Math., 60(1) (2007), 67-112.

\bibitem{YangaShiaZhu} H. Yanga, J. Shia, S. Zhu, \textit{Global existence of solutions for damped wave equations with nonlinear memory}, Applicable Analysis, 92 (2013), 1-13.

\end{thebibliography}
\end{document}